\documentclass{elsarticle}

\usepackage[margin=1in, includefoot]{geometry}       
\usepackage{amsthm, amssymb, amsmath}
\usepackage{epstopdf}
\usepackage{bm}
\usepackage{xcolor}

\newcommand{\be}{\begin{equation}}
\newcommand{\ee}{\end{equation}}
\newcommand{\bes}{\begin{equation*}}
\newcommand{\ees}{\end{equation*}}
\newcommand{\ba}{\begin{align}}
\newcommand{\ea}{\end{align}}
\newcommand{\cB}{{\mathcal B}}

\newcommand{\cF}{{\mathcal F}}

\newcommand{\cL}{{\mathcal L}}
\newcommand{\cM}{{\mathcal M}}

\newcommand{\cP}{{\mathcal P}}

\newcommand{\cR}{{\mathcal R}}

\newcommand{\cU}{{\mathcal U}}

\newcommand{\N}{\mathbb{N}}
\newcommand{\R}{\mathbb{R}}

\newcommand{\dist}{\mathop{\rm dist}}

\newcommand{\di}{\mathop{\rm div}\nolimits}

\newtheorem{theorem}{Theorem}
\newtheorem{remark}{Remark}

\newtheorem{definition}{Definition}

\journal{}

\begin{document}


\begin{frontmatter}
\title{Nonlinear approximation of high-dimensional anisotropic analytic functions}
\author[1]{Diane Guignard\corref{cor1}%
\fnref{fn1}}
\ead{dguignar@uottawa.ca}
\author[2]{Peter Jantsch\fnref{fn1}}
\ead{peter.jantsch@wheaton.edu}
\affiliation[1]{organization={University of Ottawa, Department of Mathematics and Statistics},
addressline={150 Louis-Pasteur Pvt},
city={Ottawa},
state={ON},
country={Canada},
postcode={K1N 6N5}}
\affiliation[2]{organization={Wheaton College, Department of Mathematics and Computer Science},
addressline={501 College Ave},
city={Wheaton},
state={IL},
postcode={60187},
country={USA}}
\cortext[cor1]{Corresponding author}
\fntext[fn1]{DG acknowledges the support provided by the Natural Science and Engineering Research Council (NSERC, grant RGPIN-2021-04311). PJ was partially supported by NSF Fellowship DMS-1704121.}

\date{\today}

\begin{abstract}
Motivated by nonlinear approximation results for classes of parametric partial differential equations (PDEs), we seek to better understand so-called library approximations to analytic functions of countably infinite number of variables. Rather than approximating a function of interest in a single space, a library approximation uses a collection of spaces and the best space may be chosen for any point in the domain. In the setting of this paper, we use a specific library which consists of local Taylor approximations on sufficiently small rectangular subdomains of the (rescaled) parameter domain $Y:=[-1,1]^\N$.
When the function of interest is the solution of a certain type of parametric PDE, recent results~\cite{bonito2020nonlinear} prove an upper bound on the number of spaces required to achieve a desired target accuracy. In this work, we prove a similar result for a more general class of functions with anisotropic analyticity, namely the class introduced in \cite{BDGJP2020}. In this way we show both where the theory developed in \cite{bonito2020nonlinear} depends on being in the setting of parametric PDEs with affine diffusion coefficients, and also expand the previous result to include more general types of parametric PDEs.
\end{abstract}


\end{frontmatter}   

\section{Introduction} 

Polynomial and piecewise polynomial approximations are fundamental tools in numerical analysis, forming the basis for many widely used methods such as the finite element method. More recently, results on polynomial and piecewise polynomial approximation have become important for understanding model reduction techniques for parametrized partial differential equations (PDEs), which are a typical model for describing complex systems in the field of uncertainty quantification or optimization. In this work, we consider these problems through the more general framework of approximating a function 
\begin{equation*}
y \in Y \mapsto u(y) \in X,
\end{equation*}
where $X$ is a Banach space, and $Y\subset \R^d$ is the parameter domain with $d$ large or even countably infinite. To prove results which are immune to the dimension $d$, we assume in what follows that the parameters are countably infinite, namely $y=(y_1,y_2,\ldots)$, and that they have been rescaled so that $Y=[-1,1]^\N$. The finite dimensional case $y=(y_1,\dots,y_d)$ with $d<\infty$ can always be recast in this setting by considering that $u(y)$ does not depend on the variables $y_j$ for $j>d$.

Because of the dimensionality, it is often crucial to perform a model reduction (dimension reduction) for $u$. A typical model reduction method is based on introducing a {\em linear space} $X_n$, of low dimension $n$, which is tailored to provide an accurate approximation to all $u(y)$ as $y$ varies in $Y$, or equivalently, to
\be
\cM:=\{u(y) \in X:\ y\in Y\}.
\ee 
One possibility for obtaining such dimension reduction is to approximate $u$ by Banach space valued polynomials in $y$. Another common approach is to define $X_n$ to be the linear span of so-called snapshots $u(y^1), \ldots, u(y^n)\in X$ for suitably selected parameters $y^1,\ldots y^n\in Y$ and, given $y\in Y$, approximate $u(y)$ by its orthogonal projection onto $X_n$.

Recent results from~\cite{ bonito2020nonlinear, eftang2010, maday2013, zou2019adaptive} have also drawn attention to the use of {\em nonlinear} model reduction methods, which have several advantages over linear methods. First, the ability to use approximation spaces of small dimension enables one to avoid the computationally expensive process of projecting a function into a space of large dimension, which must be done online whenever the reduced model is utilized. Another advantage is in the problem of state and parameter estimation, which is one of the main motivations for the introduction of nonlinear reduced models in \cite{bonito2020nonlinear}. In this framework, there is limited information available about the solution manifold, usually in the form of linear measurements of the state, and increasing the dimension of the reduced space beyond that of the space characterizing the measurements is unhelpful; see for instance \cite{BCDDPW17}.

Although we are motivated by parametric PDE applications, we formulate and study this subject in a more abstract form as a problem in multivariate approximation. By deriving results for a general class of functions, namely the class of anisotropic analytic functions introduced in \cite{BDGJP2020} (see Definition \ref{def:Brp} below), we hope to draw attention to two points. First, we show where the previous theory from~\cite{bonito2020nonlinear} relies on being in the setting of an elliptic parametric PDE with affine diffusion coefficient. Furthermore, by expanding the result to a general class of functions, we may then apply it to some other types of parametric PDEs as soon as the corresponding solution satisfies appropriate analyticity assumptions. 

\subsection{Linear reduced models} \label{subsec:linear}
As mentioned above, there are two general approaches to finding a linear reduced model $X_n$. When the map $y\mapsto u(y)$ has a certain analyticity in $y$, the Taylor polynomial approach\footnote{Depending on the context, other polynomial representations could be used such as expansions with respect to Legendre, Chebyshev or Hermite polynomials.} makes use of the series representation

\be
\label{ps}
u(y)=\sum_{\nu\in \cF} t_\nu y^\nu, \quad t_\nu\in X.
\ee 
Here, $\cF $ denotes the set of finitely-supported sequences $\nu=(\nu_1,\nu_2,\dots)$, i.e., sequences with finitely many non-zero terms, and whose entries are nonnegative integers, and $y^{\nu}=\prod_{j\ge 1}y_j^{\nu_j}$. Under certain analyticity assumptions, quantitative bounds for the norms of the Taylor coefficients, $\|t_{\nu}\|_X$, allow one to prove that for any $\varepsilon$, there is a finite set $\Lambda=\Lambda(\varepsilon)\subset \cF$ such that 
\be
\label{eapprox}
\sup_{y\in Y} \|u(y)-\sum_{\nu\in\Lambda} t_\nu y^\nu\|_X\le \varepsilon.
\ee
The space $X_n:={\rm span}\{t_\nu : \nu\in\Lambda\}$ provides the reduced model with $n=\#(\Lambda)$. In this case, an approximation of $u(y)$ in $X_n$ is readily provided by the function 
\be
\label{rm}
u(y) \approx \hat u(y):= \sum_{\nu\in\Lambda} t_\nu y^\nu,
\ee
in other words, we approximate $u$ using the terms $y^\nu$ as the coefficients of $\hat u$ in the basis $\{t_\nu\}_{\nu\in\Lambda}$. 
Moreover, the sets $\Lambda$ may be chosen as {\em lower sets},  which are defined by the property
\bes
\hbox {if} \quad \nu\in \Lambda, \quad \hbox{then} \quad \mu\in\Lambda \quad \hbox{whenever} \quad \mu_j\le \nu_j, \quad j=1,2, \dots.
\ees
 
A second approach to finding a reduced model is to judiciously select certain {\it snapshots} $u(y^1),\dots, u(y^n) \in X$ of $u$, for instance via a greedy procedure, and use the space spanned by the snapshots as the reduced model, i.e., $X_n:=\mathrm{span}\{u(y^1),\dots, u(y^n)\}$. In this case, the approximation of $u(y)$ in $X_n$ is defined to be 
\begin{equation*}
u(y) \approx \tilde u(y):=P_{X_n}u(y),
\end{equation*}	
where $P_{X_n}u(y)$ denotes the orthogonal projection of $u(y)$ onto $X_n$. We will comment on a few of the advantages and drawbacks of these methods in Section \ref{sec:ellipticPDE} below.
 
\subsection{Nonlinear reduced models and library approximation} \label{subsec:nonlinear}

In many contexts, numerical methods based on {\em nonlinear} approximation perform better than their linear counterpart, in the sense that they achieved a prescribed accuracy with fewer degrees of freedom \cite{Dnonlinear}. This motivates us to consider replacing the linear reduced model $X_n$ by a nonlinear space $\Sigma_n$ depending on $n$ parameters. We call such a space $\Sigma_n$ a {\it nonlinear reduced model}. This idea has already been suggested and studied in certain settings; see e.g., \cite{bonito2020nonlinear, eftang2010, maday2013, zou2019adaptive}. 

The nonlinear reduced models studied in this paper can be placed into the form of what is sometimes called {\it library approximation}. Given a Banach space $X$, a library $\cL$ is a finite collection of affine spaces $L_1:=x_1+X_1,\dots,L_N:=x_N+X_N$, where each $X_j$ is a linear space of dimension at most $m$, and each $x_j\in X$, $j=1,\dots,N$. We set each $X_j=\{0\}$ in the case $m=0$. For an element $x\in X$, the error of approximation of $x$ by the library $\cL$ is
\be
\label{xerror}
E(x,\cL):=\inf_{L\in\cL} \dist(x,L)_X.
\ee
In other words, given $x$, we choose the best of the affine spaces $L_j=x_j+X_j$, $j=1,\ldots,N$, to approximate $x$. Given a library $\cL$ and a compact set $K\subset X$, we define the worst-case error
\be
\label{liberror}
E_\cL(K):=\sup_{x\in K} E(x,\cL).
\ee
For a given a class of functions $K$, the goal is then to build a library with $m$ small, in comparison to the dimension $n$ used in linear models $X_n$, while retaining the accuracy of the reduced model and keeping the dimension of the library $N$ moderate.

We denote by $\cL_{m,N}=\cL_{m,N}(X)$ the collection of all libraries $\cL = \{ L_1,\ldots,L_N\}$ containing $N$ affine spaces of dimension at most $m$. If we fix the values of $m$ and $N$, then the best performance of a library with these fixed values is
\be
\label{mNwidth}
d_{m,N}(K):=\inf_{\cL \in\ \cL_{m,N}}E_\cL(K).
\ee
We call $d_{m,N}$ the {\it library width} of $K$. This definition slightly differs from that introduced in \cite{Tem} in which the spaces $L_j$ are taken to be linear instead of affine.
 
Library widths include the two standard approximation concepts of widths and entropy. Recall that if $K$ is a compact set in a Banach space $X$, then its Kolmogorov $m$-width is
\be
\label{nwidth}
d_m(K):= d_m(K)_X:=\inf_{\dim(W)=m}\dist(K,W)_X,
\ee
where the infimum is taken over all linear spaces $W$ of dimension $m$. Thus the Kolmogorov $m$-width of $K$ is the smallest error that can be obtained by approximating $K$ with linear spaces of dimension $m$. It follows that we can bound the library width $d_{m,1}(K)_X$ between Kolmogorov widths by 
\be
d_{m+1}(K)\leq d_m(K_0)=d_{m,1}(K) \leq d_{m}(K),
\ee
where $K_0=K-x_0$ for some suitable $x_0\in X$. At the other extreme,
\be
d_{0,2^n}(K)=\varepsilon_n(K),
\ee
where $\varepsilon_n(K)$ is the $n$th entropy number of $K$: that is, the smallest number $\varepsilon$ such that $K$ can be covered by $2^n$ balls of radius $\varepsilon$ in $X$. 

\subsection{Application to Parametric PDEs} \label{sec:ellipticPDE}

As mentioned before, one of the main motivations in considering the approximation of high-dimensional Banach or Hilbert space-valued functions is the solution of parametric PDEs, which take the general form
\begin{equation} \label{abstract_PPDE}
\cP(u,y)=0,
\end{equation}
where $y$ ranges over some parameter domain $Y$ and $u=u(y)$ is the corresponding solution assumed to be uniquely defined in some Hilbert space $V$ for every $y\in Y$. We distinguish two frameworks in which such parametric PDEs arise. In optimal design or inverse problems, the goal is to minimize an objective function involving $u(y)$. In the uncertainty quantification setting, where the parameter vector $y$ is the realization of some random vector modeling the uncertainty in the system, we are interested in computing statistics of $u(y)$. Both settings require the solution of the parametric PDE \eqref{abstract_PPDE} for many different values of $y\in Y$, where again the rescaled parameter domain is assumed to be $Y=[-1,1]^{\mathbb N}$. To relate the following to the general discussion above, let $X=V$ and define $K=\cM:=\{u(y):\ y\in Y\}$ to be the solution manifold of the PDE.

There is a rigorous theory that quantifies the approximation performance of linear reduced models for parametric PDEs; see \cite{CD} for a summary of known results. The theory is most fully developed in the case of elliptic PDEs of the form 
\be 
\label{elliptic}
-\di(a\nabla u) = f,
\ee
set on a physical domain $D\subset \R^k$ (typically $k=1,2,3$), with, e.g.,  Dirichlet boundary conditions $u_{|\partial D}=0$, and where the diffusion function $a$ has an affine representation of the form
\begin{equation} 
\label{affine}
a(y)=\bar{a}+\sum_{j\geq 1}y_j\psi_j,
\end{equation}
for some given functions $\bar{a}$ and $(\psi_j)_{j\geq 1}$ in $L^{\infty}(D)$. These functions are assumed to satisfy the condition
\begin{equation} \label{hyp:UEA*}
\left\|\frac{\sum_{j\geq 1}|\psi_j|}{\bar{a}}\right\|_{L^{\infty}(D)}<1,
\end{equation}
which is equivalent to the following \emph{Uniform Ellipticity Assumption} (UEA): there exist $0<a_\mathrm{min} \leq a_\mathrm{max} <\infty$ such that
\begin{equation} \label{hyp:UEA}
 0< a_\mathrm{min} \leq a(y)\leq a_\mathrm{max} <\infty, \quad y\in Y.
\end{equation}
\vskip .1in
\noindent
Lax-Milgram theory then ensures that whenever $f\in V'=H^{-1}(D)$, for each $y\in Y$, the corresponding solution $u(y)$ is uniquely defined in the Hilbert space $V:=H_0^1(D)$ endowed with the norm $\|v\|_V:=\|\nabla v\|_{L^2(D)}$, $v\in V$.

Under these assumptions on the parametrized input data of the PDE, it is known that $u$ admits an analytic extension onto certain complex polydiscs or so-called filled-in Bernstein polyellipses that contain $Y$ (see \cite{CD}). In other words, $u$ has a certain anisotropic analyticity, dictated by the radius of the polydiscs or the length of the semi-axis of the polyellipses, respectively, and is therefore amenable to approximation by polynomials.

For this affine parametric model (and some related to it), recent results show that there is a numerical advantage in the Taylor coefficient approach to approximating $u$.
More specifically, it is sometimes possible to find a priori a suitable lower set $\Lambda$ by exploiting the parametric form of the diffusion coefficients~\cite{BCM}, and bounds on the cardinality of $\Lambda$ needed to reach a prescribed accuracy are available; see for instance \cite{CD,BDGJP2020}. This avoids computationally expensive search algorithms that are a component of greedy reduced basis selections.  Moreover, the (Galerkin) projection involved in greedy algorithms is not needed, saving the cost of solving a linear system with a dense $n\times n$ matrix.
On the other hand, greedy procedures have the advantage that they are provably near-optimal for finding a linear space to approximate $u$, in the sense that their convergence rates are similar to those of the optimal linear spaces for approximating $\cM$ \cite{BCDDPW}. Numerical experiments show that for a prescribed target accuracy, the greedy generated spaces that meet this accuracy are of significantly lower dimension then their polynomial counterparts~\cite{bonito2020nonlinear}.
However, as mentioned above, the (offline) computational costs to constructing a reduced basis via a greedy algorithm becomes prohibitive as the target accuracy $\varepsilon$ gets small (or $n$ is getting large); we refer to \cite{CD} for a detailed analysis. This high computational cost can be alleviated if one uses random training sets, see \cite{CDD}, in which case the error bounds are no longer certified but hold with high probability.

In the case of nonlinear reduced models for parametric PDEs, a library $\cL$ would then consist of affine spaces
\be
\label{affinespace}
L_i:=u_i+V_i,
\ee
where each $u_i\in V$ and each $V_i\subset V$ has dimension at most $m$. Then, the best approximation to $u(y)$ from $L_i$ is 
\be
\label{ba}
u_i +P_{V_i}(u(y)-u_i),
\ee
where $P_{V_i}$ is the $V$-orthogonal projection onto $V_i$.
In this context, when presented with a parameter $y$ for which we wish to compute an approximation to $u(y)$, the choice of which space $L_i$ to use from a given library $\cL$ could be decided in several ways. One possibility would be to find a computable upper bound for $\dist(u(y),L_i)_V=\|u-u_i-P_{V_i}(u(y)-u_i)\|_V$, namely a quantity only involving the projection $P_{V_i}(u(y)-u_i)$ and input data, and choose the value of $i$ that minimizes this surrogate quantity. This approach would however require the computation of the projection $P_{V_i}(u(y)-u_i)$ onto each space $V_i\in\cL$ which can be prohibitive when the cardinality of the library is large. Another procedure, and the one considered in this paper, involves building an a priori partition of the parameter domain $Y$ into cells $Q_i$, and constructing an affine space $L_i$ for each cell. Then the choice of $L_i$ for approximating $u(y)$ is determined by the cell $Q_i$ containing $y$.
One of the motivations for using library approximations with a small value of $m$ is to control the offline costs needed to construct each affine space $L_i$ in the collection. In addition, having a collection of low-dimensional spaces can also be beneficial in terms of online costs, as each query of the parameter-to-solution map $y\mapsto u(y)$ is computationally cheap. We also mention that keeping $m$ small is in fact required is some contexts, for instance when estimating the state from data observations. Indeed, in this setting the dimension of the reduced spaces are limited by the number of measurements of the state.

There are many parametric PDE problems outside of this theory for elliptic diffusion equations with coefficients of the form~\eqref{affine}, which nonetheless has certain anisotropic analyticity properties. Therefore, one goal of this paper is to prove results for nonlinear approximation in these cases. In the following sections, we do this by considering general approximation classes of anisotropic analytic functions, without regard for any specific parametric PDE setting.

\subsection{Outline} \label{sec:outline}
The rest of this paper is organized as follows. In Section \ref{sec:anisoclass} we introduce the model class of anisotropic analytic functions with Definition \ref{def:Brp}. With Theorem \ref{thm:global}, we then give a global error estimate for the (Taylor) polynomial approximation error in the $L^{\infty}(Y,X)$ norm, and discuss sufficient conditions for proving local error estimates. In Section \ref{sec:interp}, we introduce results from interpolation theory that are used to prove one of the local error estimate given in Section \ref{sec:pp_general}. In the latter, we give two upper bounds for the error in a subdomain $Q\subset Y$; see Theorems \ref{thm:local_v1} and \ref{thm:local_v2}. Although both estimates provide the same convergence rate with respect to the dimension of the (local) reduced space, the upper bound derived in Theorem \ref{thm:local_v1} requires that the sequence $\rho=(\rho_1,\rho_2,\ldots)$ characterizing the analytic anisotropy of the function to be approximated satisfies $(\rho_j^{-1})_{j\ge 1}\in\ell_1(\mathbb{N})$. This is a restrictive assumption that is not needed for the bound in Theorem \ref{thm:local_v2} which involves a modified sequence $\kappa=(\rho_1^{\theta},\rho_2^{\theta},\ldots)$ with $0<\theta\le 1$. Section \ref{sec:comparison} contains a comparison of the error estimates based on the original and modified sequences $\rho$ and $\kappa$. The local bound of Theorem \ref{thm:local_v2} is then used in Section \ref{sec:UB_size_library} to derive an upper bound on the dimension of the library needed to achieved a prescribed accuracy using spaces of a fixed dimension on each subdomain. Finally, concluding remarks are made in Section \ref{sec:conclusion}.

\subsection{Notation}
For sequences $a=(a_j)_{j\ge 1}$, $b=(b_j)_{j\ge 1}$ and $c=(c_j)_{j\ge 1}$, we will assume the following operations act elementwise and write
$$a\pm b = (a_1\pm b_1,a_2\pm b_2,\ldots), \quad |a|=(|a_1|,|a_2|,\ldots), \quad a+cb=(a_1+c_1b_1,a_2+c_2b_2,\ldots),$$
and
$$a^b=\prod_{j\ge 1}a_j^{b_j}, \quad a!=\prod_{j\ge 1}a_j!.$$
Moreover, for any multi-index set $\Lambda\subseteq\cF$ containing the zero sequence $0$, we write $\Lambda^*:=\Lambda\setminus\{0\}$.

\section{Anisotropic analyticity and approximation} \label{sec:anisoclass}

In this section, we recall the classes of anisotropic analytic functions introduced in \cite{BDGJP2020} and derive a (global) approximation error estimate. Throughout this paper  $\rho=(\rho_1,\rho_2,\dots)$ denotes a sequence of positive real numbers larger than one (i.e., $\rho_j>1$ for all $j\ge 1$) which satisfies $\lim_{j\to\infty}\rho_j=\infty$.
We recall the Banach space $\ell_\infty(\N)$ of all bounded {\it complex valued} sequences $(z_j)_{j\ge 1}$, with its usual norm $\|z\|_{\ell_\infty{(\N)}}:=\sup_{j\ge 1} |z_j|$, and let $\cU$ denote the unit ball of $\ell_\infty(\N)$.

In what follows, we are interested in representing a function $u \in L^\infty(Y,X)$ by a Taylor series expansion
\be
\label{TE}
u(y)=\sum_{\nu\in\cF} t_\nu y^\nu, \quad t_{\nu}\in X.
\ee
An important question is in which sense this Taylor series converges. We say that the convergence is {\it uniform unconditional} if any rearrangement of the terms in the series in \eqref{TE} converges uniformly in the space $X$. As noted in \cite[\S 3.1]{CD}, such convergence is ensured whenever $(\|t_\nu\|_X)_{\nu\in\cF}$ is in $\ell_1(\cF)$. Moreover, if $u$ satisfies a so-called \emph{truncation property}, then $u$ coincide with its Taylor series, see \cite[Proposition 2.1.5]{Zech} for details. The class of functions $\cB_{\rho,p}$ introduced in \cite{BDGJP2020} then consists of $X$-valued functions with convergent Taylor series that has certain anisotropy dictated by the sequence $\rho$. The precise definition is as follows.

\begin{definition} \label{def:Brp}
For any $0<p\le \infty$, we define the space $\cB_{\rho,p}$ as the set of all $u\in L^\infty(Y,X)$  which admit a representation 
\bes
\label{Brep}
u(y)=\sum_{\nu\in\cF} t_\nu y^\nu,\quad y\in Y,
\ees
with the convergence of the series uniform unconditional on $Y$, and where the $t_\nu=t_\nu(u)\in X$ are unique and satisfy
\be
\label{Brhopnorm_bounded}
\|u\|_{\cB_{\rho,p}}<\infty,
\ee
where
\be
\label{Brhopnorm}
\|u\|_{\cB_{\rho,p}}:= \left\{\begin{array}{ll}
\left( \sum_{\nu\in \cF} [\rho^\nu \|t_\nu\|_X]^p\right)^{1/p}= \| (\rho^\nu\|t_\nu\|_X)_{\nu\in\cF}\|_{\ell_p(\cF)} & \mbox{if } 0<p<\infty \\
& \\
\sup_{\nu\in\cF}\rho^\nu\|t_{\nu}\|_X & \mbox{if } p=\infty.
\end{array}
\right.
\ee
\end{definition}

Note that these classes get smaller as $p$ decreases, i.e., $\cB_{\rho,p}\subset \cB_{\rho,q}$ when $p\le q$.

\begin{remark}
We could define anisotropic spaces using other sequence norms in place of $\ell_p$ norms \eqref{Brhopnorm}, for instance Lorentz space norms. Moreover, different classes of analytic anisotropic functions could be defined replacing the Taylor basis $y^\nu$, $\nu\in\cF$, by other polynomial bases, a relevant example being a basis constituted of Legendre polynomials \cite{CCS}.
\end{remark} 

The goal of this work is to study nonlinear approximation for the model classes $\cB_{\rho,p}$, and more specifically to investigate the approximation of $u\in\cB_{\rho,p}$ by a library of $X$-valued piecewise (Taylor) polynomials. The general idea is the following: we fix a target accuracy $\varepsilon>0$ and an integer $m\ge 0$, and form a partition $ \{Q_i\}_{i=1}^N$ of $Y=[-1,1]^{\N}$ consisting of $N:=N(\varepsilon,m)$ subdomains $Q_i$. 
These $Q_i$ are chosen such that there is a $X$-valued polynomial of the form
\begin{equation*}
P_i(y)=\sum_{\nu\in\Lambda_i} t_{\nu,i}(y-\bar y^i)^\nu, \qquad \mbox{satisfying} \qquad \sup_{y\in Q_i}\|u(y)-P_i(y)\|_X\le \varepsilon,
\end{equation*}
where $\bar y^i\in Q_i$, $\Lambda_i$ is a (lower) set of cardinality $m+1$, and $t_{\nu,i}\in X$ for $\nu\in\Lambda_i$.
In other words, since $\Lambda_i$ is lower and thus contains the zero sequence, to each subdomain $Q_i$ we associate an affine space of the form \eqref{affinespace}
$$L_i:=t_{0,i} + X_i, \quad X_i:={\rm span}\{t_{\nu,i}: \,\, \nu\in\Lambda_i^*\}, \, \dim(X_i)=m.$$
For ease of both mathematical analysis and practical computation, we consider a tensor product partition of $Y$, yielding a covering made of hyperrectangles as defined in \eqref{defQ} below. This process, sketched in the proof of Theorem \ref{tcount_general_PDE}, is done as in our previous work \cite{bonito2020nonlinear}.

The goal is then to give an upper bound on the required number of subdomains, i.e., on the size of the library, and provide a recipe to built a suitable partition. The main ingredient for reaching that goal is the derivation of local error estimates, namely the estimation of the error in the $L^{\infty}(Q_i,X)$ norm between $u\in\cB_{\rho,p}$ and its truncated Taylor series about some $\bar y^i\in Q_i$. Unlike the PDE setting, such local error estimates cannot be deduced directly from a global error estimate using a simple shifting and scaling argument. Nonetheless, we first derive an error estimate on the whole parameter domain $Y$ for the error between $u$ and a truncated Taylor series about the origin.

\begin{theorem} \label{thm:global} 
Let $u\in \cB_{\rho,p}$ for some $1\le p\le \infty$, and let $p'$ denote the conjugate of $p$, i.e., $1/p+1/p'=1$. Assume that the sequence $\rho$ satisfies $\rho_j>1$ for all $j\ge 1$, and $(\rho_j^{-1})_{j\ge 1} \in \ell_q(\mathbb{N})$ for some $0< q < p'$. Then for any $m\geq0$, there is a lower set $\Lambda$ with $\#(\Lambda)=m+1$ such that the $X$-valued polynomial $P(y):=\sum_{\nu\in\Lambda}t_\nu y^\nu$, $t_{\nu}:=\partial_{\nu} u(0)/\nu!,$ satisfies
\begin{equation} \label{est:global_v2}
\|u-P\|_{L^{\infty}(Y,X)}\leq C(\rho,q)\|u\|_{\cB_{\rho,p}}\|(\rho_j^{-1})_{j\ge 1}\|_{\ell_q}(m+1)^{-r}, \quad r:=-1+\frac{1}{p}+\frac{1}{q}>0,
\end{equation}
where $C(\rho,q)>0$ depends only on $q$ and the sequence $\rho$.
\end{theorem}

\begin{proof}
We start by showing that there exists a polynomial $P$ with $m+1$ terms such that
\begin{equation} \label{est:global}
\|u-P\|_{L^{\infty}(Y,X)}\leq \|u\|_{\cB_{\rho,p}}\|(\rho^{-\nu})_{\nu\in\cF^*}\|_{\ell_q}(m+1)^{-r}.
\end{equation}
Let $\Lambda$ be a lower set of indices $\nu \in \cF$ corresponding to the $m+1$ largest terms $\rho^{-\nu}$. Such a set, which may not be unique, can be obtained by a proper handling of the possible ties in the $\rho^{-\nu}$; see \cite{BDGJP2020}.
Recalling that $p'$ denote the conjugate of $p$, using the H\"older inequality we have for any $y\in Y$
\begin{equation} \label{step1}
\|u(y)-P(y)\|_X\leq \sum_{\nu\notin\Lambda}\|t_{\nu}\|_X=\sum_{\nu\notin\Lambda}\|t_{\nu}\|_X\rho^{\nu}\rho^{-\nu}\leq \|u\|_{\cB_{\rho,p}}\left[\sum_{\nu\notin\Lambda}\rho^{-\nu p'}\right]^{\frac{1}{p'}}.
\end{equation}
Moreover, we easily get
\begin{equation} \label{step2}
\sum_{\nu\notin\Lambda}\rho^{-\nu p'}=\sum_{\nu\notin\Lambda}\rho^{-\nu (p'-q)}\rho^{-\nu q}\leq \left[\sup_{\nu\notin\Lambda}\rho^{-\nu (p'-q)}\right]\sum_{\nu\notin\Lambda}\rho^{-\nu q}.
\end{equation}
We now let $(\gamma_k)_{k\geq1}$ be a non-increasing rearrangement of the sequence $(\rho^{-\nu})_{\nu\in\cF}$. We note that $\gamma_1$ will always correspond to $\rho^{-0}$ due to the fact that $\rho_j > 1$ for all $j\ge 1$. Then we have
\begin{equation}
\sup_{\nu\notin\Lambda}\rho^{-\nu q} = \gamma^q_{m+2} \leq (m+1)^{-1} \sum_{k=2}^{m+2} \gamma_k^q \leq (m+1)^{-1} \sum_{k\geq 2} \gamma_k^q = (m+1)^{-1} \sum_{\nu \neq 0} \rho^{-q\nu}
\end{equation}
which implies
\begin{equation} \label{step3}
\sup_{\nu\notin\Lambda}\rho^{-\nu}\leq (m+1)^{-\frac{1}{q}}\|(\rho^{-\nu})_{\nu\in\cF^*}\|_{\ell_q}.
\end{equation}	
Inserting \eqref{step3} in \eqref{step2} and then in \eqref{step1} we get
\begin{eqnarray*}
\|u(y)-P(y)\|_X & \leq & \|u\|_{\cB_{\rho,p}}\left[\left\{(m+1)^{-\frac{1}{q}}\|(\rho^{-\nu})_{\nu\in\cF^*}\|_{\ell_q}\right\}^{p'-q}\right]^{\frac{1}{p'}}\left[\sum_{\nu\notin\Lambda}\rho^{-\nu q}\right]^{\frac{1}{p'}} \\
	& \leq & \|u\|_{\cB_{\rho,p}}(m+1)^{-\frac{p'-q}{qp'}}\left(\sum_{\nu\neq 0}\rho^{-\nu q}\right)^{\frac{p'-q}{qp'}}\left(\sum_{\nu\neq 0}\rho^{-\nu q}\right)^{\frac{1}{p'}} \\
	& = & \|u\|_{\cB_{\rho,p}}(m+1)^{-\frac{1}{q}+\frac{1}{p'}}\|(\rho^{-\nu})_{\nu\in\cF^*}\|_{\ell_q}.
\end{eqnarray*}
To conclude the proof, it remains to bound the $\ell_q$ norm of $(\rho^{-\nu})_{\nu\in\cF^*}$ by that of $(\rho_j^{-1})_{j\ge 1}$. Thanks to	\cite[Theorem 3.1]{bonito2020nonlinear} we have	
\begin{equation} \label{def:cst_equiv_norms}
\|(\rho^{-\nu})_{\nu\in\cF^*}\|_{\ell_q} \le C(\rho,q) \|(\rho_j^{-1})_{j\ge 1}\|_{\ell_q},
\end{equation}
where
\begin{equation}\label{def:Cdeltaq}
C(\rho,q) := \beta^{\frac{1}{q}}\exp\big(\frac{\beta}{q}\|(\rho_j^{-1})_{j\ge 1}\|_{\ell_q}^q\big), \qquad \beta:=-\ln(1-\rho_{\min}^{-q})\rho_{\min}^q, \quad \rho_{\min} :=\min_{j\ge 1}\rho_j>1,
\end{equation}
and the proof is complete.
\end{proof}

\begin{remark}
The error estimate \eqref{est:global} is a straightforward extension of \cite[Equation (3.13)]{bonito2020nonlinear} to the general case $1\le p \le \infty$, and its proof is given here for convenience. In \cite[Equation (3.13)]{bonito2020nonlinear}, which corresponds to the case $p=2$, the constant $C_{\delta}$ is an upper bound for $\|u\|_{\cB_{\rho,2}}$ obtained using the fact that $u$ is the solution to the elliptic parametric PDE \eqref{elliptic}; see \cite{BCM} for a proof of this \emph{a priori} estimate.
\end{remark}

\begin{remark}
Note that we could have considered the upper bound \eqref{est:global} with the norm $\|(\rho^{-\nu})_{\nu\in\cF}\|_{\ell_q}$ instead of $\|(\rho^{-\nu})_{\nu\in\cF^*}\|_{\ell_q}$, namely without excluding the zero sequence, in which case $(m+1)^{-r}$ can be replaced by $(m+2)^{-r}$. The reason for having $\cF^{*}$ is to be able to bound $\|(\rho^{-\nu})_{\nu\in\cF^*}\|_{\ell_q}$ by $\|(\rho_j^{-1})_{j\ge 1}\|_{\ell_q}$ which is needed for the tensor product partition of $Y$ considered in this work. Indeed, the cover of the parameter domain $Y$ will be obtained as in \cite{bonito2020nonlinear} by partitioning some of the directions $j=1,2,\ldots$, namely the ones contributing the most to the error. This information is encoded in the sequence $(\rho_j)_{j\ge 1}$, which controls the anisotropy of $u\in\cB_{\rho,p}$: a larger $\rho_j$ (and thus a smaller $\rho_j^{-1}$) indicates a smaller influence of the variable $y_j$ on $u$.
\end{remark}

As mentioned above, our library consists of local Taylor approximations and we want to compute an upper bound on $N(\varepsilon,m)$, the cardinality of the library for a given a target accuracy $\varepsilon$ and a dimension $m$. To do this, we derive a local version of \eqref{est:global_v2} in Theorem~\ref{thm:global}, i.e., we show a similar error bound for the Taylor series coefficients around an arbitrary point $\bar y \in Y := [-1,1]^\N$. The general idea is this: assume again that we are in the setting of Theorem \ref{thm:global} above, with $(\rho_j^{-1})_{j\ge 1}\in \ell_q(\mathbb{N})$ and $\rho_j>1$ for $j\ge 1$. Suppose then that $Q_\lambda(\bar y) \subset Y$ is a hyperrectangle centered at $\bar y$ with side-lengths $2\lambda_j$, $j\ge 1$, i.e.
\be
\label{defQ}
Q_\lambda(\bar y):=\{y\in \mathbb{R}^{\mathbb{N}}: \quad |y_j- \bar y_j|\leq \lambda_j \quad \forall j\geq 1\},
\ee
where $\lambda_j\leq 1-|\bar y_j|$ so that $Q_\lambda(\bar y) \subset Y$. We would then like to show that for any $Q:=Q_{\lambda}(\bar y)$ as in \eqref{defQ} and any $m\ge 0$, there is a polynomial $P_Q$ with $m+1$ terms such that
\be \label{convcor1}
\sup_{y\in Q}\|u(y)-P_Q(y)\|_X\le C \|(\tilde \rho_j^{-1})_{j\ge 1}\|_{\ell_q}(m+1)^{-r}, \quad r = -1 + \frac{1}{q} + \frac{1}{p},
\ee
where $C = C(u,\rho, p, \bar y)$ is some positive constant and
\be
\label{def:rhotilde}
\tilde \rho_j := \frac{\rho_j-|\bar y_j|}{\lambda_j}, \quad j\ge 1.
\ee
Once the local error estimate \eqref{convcor1} is established, then for any given center $\bar y\in Y$ and integer $m\ge 0$, we can choose a sequence $\lambda=(\lambda_j)_{j\ge 1}$ so that the truncated Taylor series of $u$ about $y=\bar y$ with $m+1$ terms achieves the target accuracy $\varepsilon$ on $Q_{\lambda}(\bar y)$. Indeed, the smaller $\lambda_j$ the larger $\tilde \rho_j$ and thus the smaller the upper bound \eqref{convcor1}.

If we define the function $v$ via $v(y) = u(\bar y + \lambda y)$, $y\in Y$, we then infer that the local error estimate \eqref{convcor1} is proved if we can show that $v$ satisfies the assumptions of Theorem~\ref{thm:global} with the sequence $\tilde{\rho}=(\tilde \rho_j)_{j\ge 1}$ defined in \eqref{def:rhotilde}. 
To see this, first note that
\be\label{tilderhobound}
	\tilde{\rho}_j = \frac{\rho_j - |\bar y_j|}{\lambda_j} = \rho_j \frac{1 - |\bar y_j|/\rho_j}{\lambda_j} \geq \rho_j \frac{1 - |\bar y_j|/\rho_j}{1 - |\bar y_j|} \geq \rho_j,
\ee
so that if $(\rho_j^{-1})_{j \geq1} \in \ell_q(\mathbb{N})$ then $(\tilde{\rho}_j^{-1})_{j \geq1} \in \ell_q(\mathbb{N})$ for the same value of $q$. We also have the relationship $\partial_\nu v(0) = \lambda^\nu \partial_\nu u(\bar y)$. Hence, if we want the Taylor series approximation of $u$ centered at $\bar y$ to achieve the same order of approximation $(m+1)^{-r}$ in $Q_\lambda(\bar y)$, we need to show
\be\label{goal}
	\sum_{\nu\in\cF} \left[ \left\| \frac{\partial_\nu v(0)}{\nu!} \right\|_X \tilde{\rho}^\nu \right]^p = \sum_{\nu\in\cF} \left[ \left\| \frac{\partial_\nu u(\bar y)}{\nu!} \right\|_X (\rho - |\bar y|)^\nu \right]^p < \infty.
\ee
Establishing~\eqref{goal}, which is used to prove the first local estimate (see Theorem~\ref{thm:local_v1}), is accomplished in \S\ref{sec:pp_general}.

When $u$ is the solution to an elliptic parametric PDE of the form \eqref{elliptic} with affine diffusion coefficient as in \eqref{affine}, we know assumptions under which the solution map $u$ satisfies the hypothesis of Theorem~\ref{thm:global} with $p=2$. Then because of the affine structure, the shifted and scaled function $v$ will satisfy a similar parametric PDE, and thus we may apply Theorem~\ref{thm:global}. Our previous work~\cite{bonito2020nonlinear} uses that parametric PDE theory to show that local error estimate holds. In the next section we prove that \eqref{goal} holds without assuming that $u$ is the affine parametric PDE solution map, namely when all we know about $u$ is that it belongs to $\cB_{\rho,p}$. To do so, we will apply results from the field of interpolation theory.

\section{Upper bound for a map between weighted sequences of Taylor coefficients} \label{sec:interp}

For an integer $1\le r\le \infty$, any positive sequence $\rho = (\rho_j)_{j\ge 1}$, and any $y\in Y$, we define the operator $T_{\rho,y}: \ell_{r}(\cF,X)\rightarrow \ell_{r}(\cF,X)$ as
\be\label{defTmultid}
	T_{\rho,y}: (v_\nu)_{\nu\in\cF} \mapsto \left( (\rho - |y|)^\nu \sum_{\mu\geq \nu} v_\mu \frac{\mu!}{(\mu-\nu)!\nu!} \rho^{-\mu} y^{\mu-\nu} \right)_{\nu\in\cF}.
\ee
For an analytic function $f:Y\rightarrow X$, this operator takes a weighted sequence of Taylor coefficients at $0$ to a weighted sequence of Taylor coefficients at $y$, i.e.,
\be \label{interpret_defTmultid}
	T_{\rho,y}: \left( \rho^\nu \frac{\partial_\nu f(0)}{\nu!} \right)_{\nu\in\cF} \mapsto \left( (\rho-|y|)^\nu \frac{\partial_\nu f(y)}{\nu!} \right)_{\nu\in\cF}.
\ee
Indeed, by definition of $T_{\rho,y}$ in \eqref{defTmultid}, the $\nu$-term of the mapped weighted sequence is
\begin{equation*}
(\rho-|y|)^{\nu}\sum_{\mu\ge\nu}\rho^{\mu}\frac{\partial_{\mu}f(0)}{\mu!}\frac{\mu!}{(\mu-\nu)!\nu!}\rho^{-\mu}y^{\mu-\nu} = (\rho-|y|)^{\nu}\frac{1}{\nu!}\sum_{\mu\ge\nu}y^{\mu-\nu}\frac{\partial_{\mu}f(0)}{(\mu-\nu)!}=(\rho-|y|)^{\nu}\frac{\partial_{\nu}f(y)}{\nu!},
\end{equation*}
where for the last equality we used the relation (obtained by a Taylor expansion)
\begin{equation*}
\partial_{\nu}f(y)=\sum_{\sigma\in\cF}\frac{\partial_{\sigma}(\partial_{\nu}f(0))}{\sigma!}y^{\sigma}\stackrel{\mu=\sigma+\nu}{=}\sum_{\mu\ge\nu}\frac{\partial_{\mu}f(0)}{(\mu-\nu)!}y^{\mu-\nu}.
\end{equation*}

The following theorem is the main result of this section and will be a key ingredient in the proof of the local error estimates given in the next section.
\begin{theorem}\label{interpolationmultid}
	Let $\rho=(\rho_j)_{j\ge 1}$ be a sequence such that $\rho_j>1$ for all $j\ge 1$, and let $y\in Y$. If $(|y_j|\rho_j^{-1})_{j\ge 1}\in\ell_1(\mathbb{N})$ then for any $1\leq p \leq \infty$ we have
	\be \label{ellpoperatornorm}
	\| T_{\rho,y} \|_{\ell_p(\cF,X) \rightarrow \ell_p(\cF,X)} \leq \left( \prod_{j=1}^\infty \left(1 - \frac{|y_j|}{\rho_j}\right)^{-1} \right)^{1 - 1/p}.
	\ee
\end{theorem}
\begin{proof}	
We first show that there exist two positive constants $M_1,M_{\infty}$ such that the map $T_{\rho,y}$ defined in~\eqref{defTmultid} satisfies
\be \label{boundT}
\| T_{\rho,y} \|_{\ell_r(\cF,X) \rightarrow \ell_r(\cF,X)} \leq M_r, \qquad r=1,\infty.
\ee
Starting with the case $r=1$, we let $v=(v_{\nu})_{\nu\in\cF} \in \ell_1(\cF,X)$. From the definition of $T_{\rho,y}$ and the triangle inequality, we see that
\begin{align*}
	\| T_{\rho,y}v \|_{\ell_1(\cF,X)} & = \sum_{\nu\in\cF} \| (T_{\rho,y}v)_\nu \|_X \\
		& = \sum_{\nu\in\cF} \left\| \sum_{\mu\geq\nu} v_\mu \frac{\mu!}{(\mu - \nu)!\nu!} y^{\mu-\nu}\rho^{-\mu}(\rho-|y|)^\nu \right\|_X \\
		& \leq \sum_{\nu\in\cF} \sum_{\mu\geq\nu} \|v_\mu \|_X \frac{\mu!}{(\mu - \nu)!\nu!} |y|^{\mu-\nu}\rho^{-\mu}(\rho-|y|)^\nu.
\end{align*}
The latter series may be reordered as
\bes
	\sum_{\nu\in\cF} \sum_{\mu\geq\nu} \|v_\mu \|_X \frac{\mu!}{(\mu - \nu)!\nu!} |y|^{\mu-\nu}\rho^{-\mu}(\rho-|y|)^\nu
		= \sum_{\mu\in\cF} \| v_\mu \|_X \rho^{-\mu} \left( \sum_{\nu\leq \mu} \frac{\mu!}{(\mu - \nu)!\nu!} |y|^{\mu-\nu}(\rho-|y|)^\nu \right).
\ees
By the binomial theorem, the inner sum is equal to $\rho^\mu$, so putting together the previous string of inequalities yields
\begin{equation}\label{eqn:operator_M1}
	 \| T_{\rho,y}v \|_{\ell_1(\cF,X)} \leq \sum_{\mu\in\cF} \| v_\mu \|_X = \| v \|_{\ell_1(\cF,X)}.
\end{equation}
This gives \eqref{boundT} for $r=1$ with constant $M_1=1$. We proceed in a similar manner for the case $r=\infty$. Letting $v=(v_{\nu})_{\nu\in\cF} \in \ell_\infty(\cF,X)$, we estimate
\begin{align*}
	\| T_{\rho,y}v \|_{\ell_\infty(\cF,X)} & = \sup_{\nu\in\cF} \| (T_{\rho,y}v)_\nu \|_X \\
		& \leq \sup_{\nu\in\cF} \sum_{\mu \geq \nu} \| v_\mu \|_X \frac{\mu!}{(\mu - \nu)!\nu!} |y|^{\mu-\nu}\rho^{-\mu}(\rho-|y|)^\nu \\
		& \leq \sup_{\mu\in\cF}\left\{ \| v_\mu \|_X \rho^\mu \right\} \sup_{\nu\in\cF} \left\{ \sum_{\mu\geq\nu} \frac{\mu!}{(\mu - \nu)!\nu!} \prod_{j=1}^\infty \left( \frac{|y_j|}{\rho_j}\right)^{\mu_j-\nu_j}\left(1-\frac{|y_j|}{\rho_j}\right)^{\nu_j} \right\}.
\end{align*}
Since $(|y_j|\rho_j^{-1} )_{j\ge 1} \in \ell_1(\N)$ by assumption, the series on the last line is equal to
\be
\prod_{j=1}^\infty \left(1 - \frac{|y_j|}{\rho_j}\right)^{-1} =: M_{\infty} <\infty,
\ee
which gives in the case $r=\infty$. Finally, \eqref{ellpoperatornorm} is obtained from \eqref{boundT} by applying the Riesz-Thorin theorem \cite{Bergh}:
\be
\| T_{\rho,y} \|_{\ell_p(\cF,X) \rightarrow \ell_p(\cF,X)} \leq M_1^{1-\theta}M_{\infty}^{\theta}, \qquad \theta=1-\frac{1}{p}.
\ee
\end{proof}

\begin{remark} \label{rem:hyp_l1}
	The assumption $(|y_j|\rho_j^{-1} )_{j\ge 1} \in \ell_1(\N)$ in Theorem~\ref{interpolationmultid} is not needed if we are only interested in the case $p=1$; see \eqref{eqn:operator_M1}.
\end{remark}

\section{Piecewise polynomial approximation for anisotropic analytic functions} \label{sec:pp_general}

We are now ready to derive a local version of Theorem \ref{thm:global}, i.e., to prove an approximation result of the form \eqref{convcor1} on subdomains of $Y$. Throughout this section, let $Q:=Q_{\lambda}(\bar y)$ be as in \eqref{defQ}, with center $\bar y\in Y$ and side-length vector $\lambda$ satisfying  $\lambda_j\leq 1-|\bar y_j|$ so that $Q \subset Y$.
A first local error estimate for the Taylor series at $\bar y$ in $Q$ is given in the following theorem.

\begin{theorem} \label{thm:local_v1}
Let $u\in \cB_{\rho,p}$ for some $1\le p\le \infty$, and assume that the sequence $\rho$ satisfies $\rho_j>1$ for all $j\ge 1$, and $(\rho_j^{-1})_{j\ge 1} \in \ell_q(\mathbb{N})$ for some $0< q \leq 1$. Then for any $m\geq0$, there is a set $\Lambda$ with $\#(\Lambda)=m+1$ such that the $X$-valued polynomial $P_{Q}(y):=\sum_{\nu\in\Lambda}\frac{\partial_\nu u(\bar y)}{\nu!} (y-\bar y)^\nu$ satisfies
\begin{equation} \label{est:local_v1}
\|u-P_{Q}\|_{L^{\infty}(Q,X)}\leq C(\rho,q)\left(\prod_{j=1}^{\infty}\left(1-\frac{|\bar y_j|}{\rho_j}\right)^{-1}\right)^{1-\frac{1}{p}}\|u\|_{\cB_{\rho,p}}\|(\tilde \rho_j^{-1})_{j\ge 1}\|_{\ell_q}(m+1)^{-r}, \quad r:=-1+\frac{1}{p}+\frac{1}{q},
\end{equation}
where
\begin{equation} \label{def:rho_tilde}
\tilde \rho_j := \frac{\rho_j-|\bar y_j|}{\lambda_j}, \quad j\ge 1,
\end{equation}
and $C(\rho,q)$ is the positive constant defined in \eqref{def:Cdeltaq}.
\end{theorem}

\begin{proof}
As before, we define the function $v(y) := u(\bar y + \lambda y)$, $y\in Y$. We will prove this theorem by showing that $v$ satisfies the assumptions of Theorem~\ref{thm:global} with the sequence $\tilde{\rho}$. 

Recall from~\eqref{tilderhobound} that $\tilde{\rho}_j \geq \rho_j,$ $j\geq1$, so that if $(\rho_j^{-1})_{j \geq1} \in \ell_q(\mathbb{N})$ then $(\tilde{\rho}_j^{-1})_{j \geq1} \in \ell_q(\mathbb{N})$ for the same value of $q$. We also have the relationship $\partial_\nu v(0) = \lambda^\nu \partial_\nu u(\bar y)$. Hence, if we want to achieve the order of approximation $(m+1)^{-r}$ by the Taylor series approximation in $Q$, we need to show that $v \in \cB_{\tilde \rho,p}$. Recalling the definition of the map $T_{\rho,\bar y}$, see in particular \eqref{interpret_defTmultid}, we have
\begin{eqnarray*}
	\sum_{\nu\in\cF} \left[ \left\| \frac{\partial_\nu v(0)}{\nu!} \right\|_X \tilde{\rho}^\nu \right]^p & = & \sum_{\nu\in\cF} \left[ \left\| \frac{\partial_\nu u(\bar y)}{\nu!} \right\|_X (\rho - |\bar y|)^\nu \right]^p \\
	& = & \|T_{\rho,\bar y}(\rho^{\nu}t_{\nu})_{\nu\in\cF}\|_{\ell_p(\cF,X)}^p \\
	& \leq & \| T_{\rho,\bar y} \|_{\ell_p(\cF,X) \rightarrow \ell_p(\cF,X)}^p \| u \|_{\cB_{\rho,p}}^p,
\end{eqnarray*}
where $t_{\nu}=\partial_{\nu}u(0)/\nu!$.

Hence, thanks to Theorem~\ref{interpolationmultid} we obtain
\begin{equation} \label{eqn:u_utilde_v1}
	\| v\|_{\cB_{\tilde \rho,p}} \leq  \left(\prod_{j=1}^{\infty}\left(1-\frac{|\bar y_j|}{\rho_j}\right)^{-1}\right)^{1-\frac{1}{p}}\|u\|_{\cB_{\rho,p}}.
\end{equation}
The right-hand side of~\eqref{eqn:u_utilde_v1} is finite since we have assumed $u\in \cB_{\rho,p}$ and 
\begin{equation*}
(\rho_j^{-1})_{j\ge 1} \in \ell_q(\mathbb{N}) \subseteq \ell_1(\mathbb{N}) \quad \Longrightarrow \quad (|y_j|\rho_j^{-1})_{j\ge 1} \in \ell_1(\mathbb{N}).
\end{equation*}
Thus $v \in \cB_{\tilde \rho,p}$, and by Theorem~\ref{thm:global}, there is a set $\Lambda$ with $m+1$ terms, such that the polynomial $P(y):=\sum_{\nu\in\Lambda}t_\nu(v) y^\nu$, $t_{\nu}(v)=\partial_{\nu}v(0)/\nu!$, satisfies
\begin{equation*}
 	\|v- P\|_{L^{\infty}(Y,X)} \leq C(\tilde\rho,q)\|v\|_{\cB_{\tilde \rho,p}} \|(\tilde \rho_j^{-1})_{j\ge 1}\|_{\ell_q}(m+1)^{-r},
\end{equation*}
where $C(\tilde\rho,q)$ is defined as in \eqref{def:Cdeltaq}. Defining $P_{Q}(y)$ by
\be
P_{Q}(y) = P((y-\bar y)/\lambda) = \sum_{\nu\in\Lambda} \frac{\partial_\nu u(\bar y)}{\nu!} (y - \bar y)^\nu,
\ee
it is easy to see that
\be \label{eqn:u_Q_v1}
\|u- P_{Q}\|_{L^{\infty}(Q,X)} \leq C(\tilde\rho,q)\|v\|_{\cB_{\tilde \rho,p}} \|(\tilde \rho_j^{-1})_{j\ge 1}\|_{\ell_q}(m+1)^{-r}.
\ee
Finally, using again that $\tilde\rho_j\ge\rho_j$ for all $j\ge 1$, and the fact that the function 
\begin{equation*}
\beta(x)=-\ln(1-x^{-1})x, \quad x\in(1,\infty),
\end{equation*}
is non-increasing, we have $C(\tilde \rho, q) \leq C(\rho, q)$. Using this with \eqref{eqn:u_utilde_v1} in \eqref{eqn:u_Q_v1} yields the desired inequality~\eqref{est:local_v1}.
\end{proof}

The drawback of Theorem \ref{thm:local_v1} is the strong assumption that $(\rho_j^{-1})_{j\ge 1}\in\ell_q(\mathbb{N})$ for some $q \leq 1$. We now show that we can get around this assumption by considering a slightly different sequence, namely $(\kappa_j)_{j\ge 1}:=(\rho_j^{\theta})_{j\ge 1}$ for some $0<\theta \le 1$, while maintaining the same convergence rate $(m+1)^{-r}$.

\begin{theorem} \label{thm:local_v2}
	Let $u\in \cB_{\rho,p}$ for some $1\le p\le \infty$, and assume the sequence $\rho$ satisfies $\rho_j>1$ for all $j\ge 1$, and $(\rho_j^{-1})_{j\ge 1} \in \ell_q(\mathbb{N})$ for some $0< q < \frac{p}{p-1}$. Let $(\kappa_j)_{j\ge 1}$ be the sequence defined by $\kappa_j:=\rho_j^{\theta}$, where $\theta:=1-\frac{q}{p'}$ with $p'$ the conjugate of $p$. Then $u\in\cB_{\kappa,1}$ and, for each $m\ge0$, there is a polynomial $P_{Q}$ with $m+1$ terms such that
	\begin{equation} \label{est:local_v2_bis}
	\|u-P_{Q}\|_{L^{\infty}(Q,X)}\leq C(\kappa,q_{\theta})\|u\|_{\cB_{\kappa,1}} \|(\tilde \kappa_j^{-1})_{j\ge 1}\|_{\ell_{q_{\theta}}}(m+1)^{-r}, \quad r:=-1+\frac{1}{p}+\frac{1}{q},
	\end{equation}
	where $q_{\theta} := \frac{q}{\theta}$ and
	\begin{equation} \label{def:kappa_tilde}
	\tilde \kappa_j := \frac{\kappa_j-|\bar y_j|}{\lambda_j}=\frac{\rho_j^{\theta}-|\bar y_j|}{\lambda_j}, \quad j\ge 1.
	\end{equation}
\end{theorem}
Note that $\theta=1-\frac{q}{p'}$, defined in Theorem \ref{thm:local_v2}, indeed belongs to $(0,1)$ whenever $p>1$, which follows from the assumption $0<q<p'=\frac{p}{p-1}$, while $\theta=1$ if $p=1$.

\begin{proof}
We first show that $u\in\cB_{\rho,p}$ implies that $u\in\cB_{\kappa,1}$. This is immediate when $p=1$, since in this case we have $\theta=1$ and $\kappa=\rho$, and thus $\|u\|_{\cB_{\kappa,1}}=\|u\|_{\cB_{\rho,p}}$. When $1<p\le \infty$, this follows from the relation
\begin{equation} \label{rel:Bk1_Brp}
\|u\|_{\cB_{\kappa,1}}= \sum_{\nu\in\cF} \kappa^\nu \|t_\nu\|_{X}=\sum_{\nu\in\cF} \rho^\nu \|t_\nu\|_{X} \rho^{(\theta-1)\nu}
\leq \|u\|_{\cB_{\rho,p}}  \Big(\sum_{\nu\in\cF} \rho^{(\theta-1)p'\nu}\Big)^{\frac 1 {p'}},
\end{equation}
where the second sum converges since $(\theta-1)p'=-q$.
Moreover, it is clear that
\be
	\sum_{\nu\in\cF} \kappa^{-q_{\theta}\nu} = \sum_{\nu\in\cF} \rho^{-q\nu} < \infty,
\ee 
so that $(\kappa^{-\nu})_{\nu\in\cF}\in\ell_{q_{\theta}}(\cF)$. Then the rest of the proof follows that of Theorem~\ref{thm:local_v1} with $\rho$, $p$, and $q$ replaced by $\kappa$, $1$, and $q_{\theta}$, respectively. In this case equation~\eqref{eqn:u_utilde_v1} gives
\begin{equation*}
\|v \|_{\cB_{\tilde \kappa, 1}} \leq \| u \|_{\cB_{\kappa, 1}},
\end{equation*}
which does not require the assumption $(\kappa_j^{-1})_{j\ge 1}\in\ell_1(\mathbb{N})$; see Remark \ref{rem:hyp_l1}.
Moreover, the convergence rate $r$ is indeed as defined in \eqref{def:kappa_tilde} since
$$r=-1+\frac{1}{1}+\frac{1}{q_{\theta}}=-1+\frac{1}{p}+\frac{1}{q}.$$
\end{proof}

If we take $\bar y=(0,0,\ldots)$ and $\lambda=(1,1,\ldots)$, then Theorem \ref{thm:local_v2} yields a global error estimate on $Q=Y$ using the sequence $\kappa$ in lieu of $\rho$, i.e., instead of \eqref{est:global_v2} we have
\begin{equation} \label{est:global_kappa_v2}
\|u-P\|_{L^{\infty}(Y,X)}\leq C(\kappa,q_{\theta})\|u\|_{\cB_{\kappa,1}} \|(\kappa_j^{-1})_{j\ge 1}\|_{\ell_{q_{\theta}}}(m+1)^{-r}
\end{equation}
for some polynomial $P=P(y)$ with $m+1$ terms. 
We compare in the next section the global error estimates \eqref{est:global_v2} and \eqref{est:global_kappa_v2}.

\begin{remark}
According to the proof of Theorem~\ref{thm:global}, the bound \eqref{est:global} is achieved by $P(y):=\sum_{\nu\in\Lambda}t_{\nu}y^{\nu}$, where $\Lambda$ is a lower set with multi-indices corresponding to the $(m+1)$ largest $\rho^{-\nu}$. Then \eqref{est:global_kappa_v2} is achieved by the same polynomial since
\begin{equation*}
\rho^{-\nu} \ge \rho^{-\mu} \quad \Longrightarrow \quad \kappa^{-\nu}=(\rho^{-\nu})^{\theta} \ge (\rho^{-\mu})^{\theta}=\kappa^{-\mu};
\end{equation*}
in other words, the ordering for the $\kappa^{-\nu}$ is the same as the one for the $\rho^{-\nu}$, and the corresponding multi-index set $\Lambda$ is the same for both sequences (provided that ties in the size of the $\rho^{-\nu}$ are handled in the same way).
\end{remark}

\section{Comparison of the global error bounds} \label{sec:comparison}

We now turn to the comparison of the global error estimates using the sequence $\rho$ versus those using $\kappa$, namely of the upper bounds for the error in the $L^{\infty}(Y,X)$ norm given in \eqref{est:global_v2} and \eqref{est:global_kappa_v2}, respectively. For ease of reference, we briefly recall the bounds on $p$ and $q$, and a few relevant definitions:
\begin{equation*}
1\le p\le \infty, \quad 0<q<\frac{p}{p-1}, \quad \kappa_j=\rho_j^{\theta}, \quad \theta=1-\frac{q}{p'}\in(0,1], \quad p'=\frac{p}{p-1}.
\end{equation*}
Note that if $p=1$, then $\theta=1$ and thus the two estimates coincide. If $p>1$, then $\theta\in(0,1)$, and we will thus consider only this case from now on.

The two error estimates consist of four terms that we compare below.

\begin{itemize}
	\item $(m+1)^{-r}$: The convergence rate $r$ is the same in both error estimates, namely $r=-1+1/p+1/q$, and thus this term is the same in both error estimates.
	\item $\|(\rho_j^{-1})_{j\ge 1}\|_{\ell_q}$ versus $\|(\kappa_j^{-1})_{j\ge 1}\|_{\ell_{q_{\theta}}}$: We have
	\begin{equation} \label{rel:norms}
	\|(\kappa_j^{-1})_{j\ge 1}\|_{\ell_{q_{\theta}}} = \|(\rho_j^{-1})_{j\ge 1}\|_{\ell_q}^{\theta}.
	\end{equation}
	Now whether $\|(\kappa_j^{-1})_{j\ge 1}\|_{\ell_{q_{\theta}}}$ is smaller than $\|(\rho_j^{-1})_{j\ge 1}\|_{\ell_q}$ depends on whether $\|(\rho_j^{-1})_{j\ge 1}\|_{\ell_q}$ is larger than one or not. Indeed, since $\theta\in(0,1)$ we have
	\begin{equation} \label{rel:norms2}
	\|(\kappa_j^{-1})_{j\ge 1}\|_{\ell_{q_{\theta}}}\le \|(\rho_j^{-1})_{j\ge 1}\|_{\ell_q} \quad \mbox{if} \quad  \|(\rho_j^{-1})_{j\ge 1}\|_{\ell_q}\ge 1
	\end{equation}
	with reverse inequality otherwise.
	\item $C(\rho, q)$ versus $C(\kappa,q_{\theta})$: Recall that these constants are defined by~\eqref{def:Cdeltaq}. Since $\rho_j>1$ for all $j\ge 1$ and $\theta\in(0,1)$, we have $\kappa_j=\rho_j^{\theta}< \rho_j$. Moreover, $\kappa_{\min}=\rho_{\min}^{\theta}$ and thus the $\beta$ in $C(\rho,q)$ and $C(\kappa,q_{\theta})$ are the same:
	\begin{equation*}
	\beta:=-\ln\left(1-\rho_{\min}^{-q}\right)\rho_{\min}^q=-\ln\left(1-\kappa_{\min}^{-\frac{q}{\theta}}\right)\kappa_{\min}^\frac{q}{\theta}=-\ln\left(1-\kappa_{\min}^{-q_{\theta}}\right)\kappa_{\min}^{q_{\theta}}.
	\end{equation*}
	Therefore, we have
	\begin{equation*}
		C(\kappa,q_{\theta}) = \beta^{\frac{\theta}{q}}\exp\left(\frac{\beta\theta}{q}\|(\kappa_j^{-1})_{j\ge 1}\|_{\ell_{q_{\theta}}}^{\frac{q}{\theta}}\right) \stackrel{\eqref{rel:norms}}{=} \beta^{\frac{\theta}{q}}\exp\left(\frac{\beta\theta}{q}\|(\rho_j^{-1})_{j\ge 1}\|_{\ell_q}^q\right) = C(\rho,q)^{\theta}\le C(\rho,q),
	\end{equation*}\\
	where for the last inequality we used the fact that $C(\rho,q)\ge 1$ since $\beta\ge 1$. Indeed, we have
	$$\ln(1-x)=-x-\frac{x^2}{2}-\frac{x^3}{3}-\ldots = -\sum_{k=1}^{\infty}\frac{x^k}{k}$$
	and thus
	$$\beta=-\ln\left(1-\rho_{\min}^{-q}\right)\rho_{\min}^q=\sum_{k=1}^{\infty}\frac{\rho_{\min}^{(1-k)q}}{k}=1+\sum_{k=2}^{\infty}\frac{\rho_{\min}^{(1-k)q}}{k}\ge 1.$$
	\item $\|u\|_{\mathcal{B}_{\rho,p}}$ versus $\|u\|_{\mathcal{B}_{\kappa,1}}$: From \eqref{rel:Bk1_Brp}, we directly have that
	\begin{equation*}
	\|u\|_{\mathcal{B}_{\kappa,1}} \le K \|u\|_{\mathcal{B}_{\rho,p}}, \quad \mbox{where } K:=\Big(\sum_{\nu\in\cF} \rho^{(\theta-1)p'\nu}\Big)^{\frac 1 {p'}}\in(1,\infty).
	\end{equation*}
	Therefore, if $K\gg 1$ then $\|u\|_{\mathcal{B}_{\kappa,1}}$ could potentially be much larger than $\|u\|_{\mathcal{B}_{\rho,p}}$. Moreover, we do not have a reversed inequality of the form $k\|u\|_{\mathcal{B}_{\rho,p}}\le \|u\|_{\mathcal{B}_{\kappa,1}}$ for some constant $k>0$, as this would imply that the spaces $\cB_{\rho,p}$ and $\cB_{\kappa,1}$ are the same.
\end{itemize}

\section{Bounds on the library size} \label{sec:UB_size_library}

In this section, we use the local error upper bound \eqref{est:local_v2_bis} to build a library consisting of piecewise Taylor polynomials. Without loss of generality, we assume that the target accuracy $\varepsilon>0$ and the dimension $m\ge 0$ are such that
\begin{equation} \label{eqn:cstr_eps_m}
C\|(\kappa_j^{-1})_{j\ge 1}\|_{\ell_q}(m+1)^{-r}>\varepsilon, \quad C:=C(\kappa,q_{\theta})\|u\|_{\mathcal{B}_{\kappa,1}},
\end{equation}
where $C(\kappa,q_{\theta})$ is the constant defined in~\eqref{def:Cdeltaq}, $\kappa_j=\rho_j^{\theta}$ and $\theta=1-\frac{q}{p'}$ as in Theorem \ref{thm:local_v2}. Indeed, if \eqref{eqn:cstr_eps_m} does not hold, then there is no need to partition the parameter domain, since thanks to the global error estimate \eqref{est:global_kappa_v2} there is a Taylor series approximation centered at the origin with $m+1$ terms satisfying the prescribed accuracy $\varepsilon$.
 
\begin{theorem} \label{tcount_general_PDE}
Let $u\in\mathcal{B}_{\rho,p}$ for some $1\le p \le \infty$, and assume that the sequence $\rho$ is nondecreasing with $\rho_1>1$ and satisfies $(\rho^{-1}_j)_{j\geq 1}\in \ell_q(\mathbb{N})$ for some $0<q<\frac{p}{p-1}$. Let $q_{\theta}=q/\theta$ with $\theta=1-q/p'$, where $p'$ is the conjugate of $p$. Finally, let $\varepsilon>0$, $m\ge 0$ and assume that \eqref{eqn:cstr_eps_m} holds. 
Then, there exists a tensor product partition of $Y$ into a collection $\cR$ of $N:=N(\varepsilon,m)$ hyperrectangles such that on each $Q\in\cR$ there is a $X$-valued polynomial $P_Q$ with $m+1$ terms such that
\begin{equation}
\label{tcount1}
\|u(y)-P_Q(y)\|_X\le \varepsilon, \quad y\in Q.
\end{equation} 
Furthermore, if $J:=J(\varepsilon,m)\ge 1$ is the smallest integer such that
\begin{equation} \label{def:J}
\sum_{j\ge J+1}\rho_j^{-q} \leq \frac{1}{2}C^{-q_{\theta}}(m+1)^{q_{\theta} r}\varepsilon^{q_{\theta}},
\end{equation}
where $C$ is the constant in \eqref{eqn:cstr_eps_m}, and
\begin{equation}
\sigma:=\left(\frac{1}{2J}\right)^{\frac{1}{q_{\theta}}}C^{-1}(m+1)^r\varepsilon,
\end{equation}
then the partition is obtained by only subdividing in the first $J$ directions, and the number of cells $N$ in this partition satisfies
\begin{equation} \label{eqn:upper_bound_N}
N\le \prod_{j=1}^{ J} \left( \sigma^{-1} | \ln( 1 - \rho_j^{-\theta}) | + C({\sigma}) \right) \quad \mbox{for some}\quad C({\sigma})\in (1,2).
\end{equation}
\end{theorem}

\begin{proof}
This proof follows closely that of \cite[Theorem 3.5]{bonito2020nonlinear}. The main difference is that we will use the local error estimate \eqref{est:local_v2_bis} to build a tensor product partition of $Y$, i.e., we use the sequence $(\kappa_j)_{j\ge 1}=(\rho_j^{\theta})_{j\ge 1}$ instead of $(\rho_j)_{j\ge 1}$, and consider the $\ell_{q_\theta}(\mathbb{N})$ norm instead of the $\ell_q(\mathbb{N})$ norm. However, the strategy for partitioning $Y$ is similar, as is the proof of the bound~\eqref{eqn:upper_bound_N}, but we give the details of the partitioning here for the reader's convenience.
	
First, since $(\rho_j)_{j\ge 1}$ is nondecreasing by assumption, $(\kappa_j)_{j\ge 1}$ is likewise nondecreasing. According to \eqref{est:local_v2_bis}, for any hyperrectangle $Q$,  centered at $\bar y$ with side-length vector $(\lambda_j)_{j\ge 1}$, a sufficient condition to have $\|u-P_{Q}\|_{L^{\infty}(Q,X)}\leq\varepsilon$ is to have
\begin{equation} \label{def:eta}
\sum_{j\ge 1}\left(\frac{\rho_j^{\theta}-|\bar y_j|}{\lambda_j}\right)^{-q_{\theta}} = \sum_{j\ge 1}\tilde\kappa_j^{-q_{\theta}} \leq C^{-q_{\theta}}(m+1)^{rq_{\theta}}\varepsilon^{q_{\theta}}=:\eta,
\end{equation}
where $C$ is defined in \eqref{eqn:cstr_eps_m}, and $P_Q$ is a polynomial with $m+1$ terms. Note that if we set set $\bar y_j=0$ and $\lambda_j=1$ for $j\ge J+1$, where $J$ is given in \eqref{def:J}, the tail of the series in~\eqref{def:eta} contributes to \emph{half of the error}, namely
\begin{equation} \label{eqn:tail}
\sum_{j\ge J+1}\tilde\kappa_j^{-q_{\theta}}=\sum_{j\ge J+1}(\rho_j^{\theta})^{-q_{\theta}}=\sum_{j\ge J+1}\rho_j^{-q}\leq \frac{1}{2}\eta.
\end{equation}
Thus, the strategy is to subdivide only in the directions $j=1,\ldots,J$, and we do this so that the remaining error is equidistributed among the first $J$ directions. This is accomplished by requiring that the center $\bar y_j$ and the sidelength $\lambda_j$ of each subinterval satisfies
\begin{equation} \label{eqn:cond_kappa_tilde}
	\tilde\kappa_j =\frac{\rho_j^{\theta}-|\bar y_j|}{\lambda_j} = \sigma^{-1}, \quad 1\le j\le J,
\end{equation}
yielding
\begin{equation} \label{eqn:partitioned_dir}
	\sum_{j=1}^J\tilde\kappa_j^{-q_{\theta}}=\sum_{j=1}^J\sigma^{q_{\theta}}=\frac{1}{2}\eta.
\end{equation}
	
To define the tensor product grid, for each coordinate direction $y_j$, $j=1,\dots,J$, we first define how we subdivide the interval $[-1,1]$ into $(2k_j+1)$ subintervals
\begin{equation*}
	I_j^{i},\quad -k_j\le i\le k_j.
\end{equation*}
We do not subdivide any of the coordinate axis when $j>J$, i.e., $k_j=0$ and $I_j^0=[-1,1]$ when $j>J$. The partition is also chosen to be symmetric and so $I_j^{-i}= -I_j^i$, $i=1,\dots ,k_j$.
  
We fix $j\in\{1,\dots,J\}$ and describe the partition of $[-1,1]$ into intervals corresponding to the $j$-th coordinate. Our first interval $I_j^0$ is centered at $\bar y_j^0=0$ and according to \eqref{eqn:cond_kappa_tilde}, we set $\lambda_j^0=\sigma\rho_j^{\theta}$ provided this number is less than one. Otherwise, when $\sigma\rho_j^\theta\geq 1$, we define $\lambda_j^0:=1$, in which case $k_j=0$ and the partition consists only of the one interval $I_j^0=[-1,1]$. Note that since $(\rho_j^\theta)_{j\geq 1}$ is nondecreasing, when this occurs it also happens for all larger values of $j$.
 
As mentioned above, the partition is symmetric with respect to the origin and so we only describe the intervals to the right of the origin. The next interval $I_j^1$, with center $\bar y_j^1$ and sidelength $\lambda_j^1$, has left endpoint the same as the right endpoint of $I_j^0$, i.e., 
\begin{equation*}
	\bar y_j^1-\lambda_j^1 = \bar y_j^0+\lambda_j^0.
\end{equation*}
Now to satisfy \eqref{eqn:cond_kappa_tilde}, we choose
\begin{equation*}
	\lambda_j^1 = \sigma(\rho_j^{\theta}-\bar y_j^1) \quad \implies \quad  \lambda_j^1=\frac{\sigma}{1+\sigma}(\rho_j^{\theta}-\bar y_j^0-\lambda_j^0).
\end{equation*}
The only exception to this definition is when the right endpoint of this interval is larger than 1. Then we recenter the interval so its left endpoint is as before and its right endpoint is 1. In this case, we would stop the process and $k_j$ would be 1.

We continue in this same way moving to the right. In general, the interval $I_j^i$ will have its left endpoint equal to the right endpoint of $I_j^{i-1}$, with center $\bar y_j^i$ and sidelength $\lambda_j^i$ which satisfy
\begin{equation}
\label{satisfies}
\lambda_j^i=\sigma(\rho_j^\theta - \bar y_j^i).
\end{equation}
As before, we rescale in the case that such a choice would give a right endpoint larger than 1 and terminate the partitioning process. It follows that the interval $I_j^i$ always satisfies
\begin{equation}
\label{satisfies1}
\lambda_j^i\le \sigma(\rho_j^{\theta}-\bar y_j^i),\quad i=0,1,\dots,k_j,
\end{equation} 
with equality except for possibly the last interval $I_j^{k_j}$. This partitioning gives a tensor product set $\cR$ of hyperrectangles $Q$. 
 
To conclude the proof, it remains to derive a bound for the number of elements $\cR$ given by
\begin{equation*}
N= \prod_{j=1}^J(2k_j+1),
\end{equation*}
namely to show that when $k_j\neq 0, j=1,\ldots,J,$ we have
\begin{equation} \label{eqn:bound_nj}
2k_j+1 \le \left(\sigma^{-1}|\ln\left( 1 - \rho_j^{-\theta} \right)|+{C({\sigma})}\right).
\end{equation}
Since the proof is technical and follows the same arguments as in \cite[Theorem 3.5]{bonito2020nonlinear}, we do not provide the details here.
\end{proof} 

Note that although the error is distributed equally among all the directions $j=1,2,\ldots,J$, see \eqref{eqn:cond_kappa_tilde}, the partition obtained using the strategy described in the proof of Theorem~\ref{tcount_general_PDE} is anisotropic in general. Indeed, the number of subintervals $2k_j+1$ in direction $j$ is controlled by $\rho_j$ through the bound \eqref{eqn:bound_nj}.

\begin{remark}
	The upper bound \eqref{eqn:upper_bound_N} is similar to \cite[Equation (3.29)]{bonito2020nonlinear} which reads
	\begin{equation} \label{eqn:upper_bound_N_prev}
	N\le \prod_{j=1}^{\tilde J} \left( \tilde\sigma^{-1} | \ln( 1 - \rho_j^{-1}) | + C({\tilde \sigma}) \right) \quad \mbox{for some}\quad C({\tilde\sigma})\in (1,2);
	\end{equation}
	notice the factor $\rho_j^{-1}$ instead of $\rho_j^{-\theta}$ in \eqref{eqn:upper_bound_N}. In \eqref{eqn:upper_bound_N_prev}, $\tilde J$ is the smallest integer such that
	\begin{equation*}
	\sum_{j\ge J+1}\rho_j^{-1}\le \frac{1}{2}\tilde C^{-q}(m+1)^{qr}\varepsilon^q,
	\end{equation*}
	where $\tilde C=C(\rho,q)C_{\delta}$ with $C(\rho,q)$ as in \eqref{def:cst_equiv_norms} and $C_{\delta}$ an upper bound for $\|u\|_{\cB_{\rho,2}}$,
	and
	\begin{equation*}
	\tilde \sigma:=\left(\frac{1}{2J}\right)^{\frac{1}{q}}\tilde C^{-1}(m+1)^r\varepsilon.
	\end{equation*}
\end{remark}

\begin{remark}
We can obtain explicit upper bounds for $N$ when considering specific sequences $\rho$. For example, if we assume that $(\rho_j)_{j\ge 1}=(Mj^s)_{j\ge 1}$ for some fixed $M>1$ and $s>1/2$, then proceeding as in \cite[Section 3.3]{bonito2020nonlinear} we have $\sigma\sim J^{-s\theta}$, $|\ln(1-\rho_j^{-\theta})|\lesssim j^{-s\theta}$, and $J\sim \lambda^{\frac{q_{\theta}}{1-sq}}$. Therefore, inserting these relations in \eqref{eqn:upper_bound_N} we infer that
$$N\lesssim J^{s\theta J}(J!)^{-s\theta} \le e^{c\lambda^{\frac{q_{
\theta}}{1-sq}}}=e^{c\lambda^{\frac{q}{\theta(1-sq)}}}, \quad \lambda:=(m+1)^r\varepsilon,$$
for some constant $c>0$ independent of $q$.
\end{remark}

\section{Conclusion} \label{sec:conclusion}

In this article, we extend the nonlinear reduced model introduced in \cite{bonito2020nonlinear} to approximate more general high-dimensional functions, namely to approximate the class $\cB_{\rho,p}$ of \emph{anisotropic analytic} functions studied in \cite{BDGJP2020}. The nonlinear reduced model, which belongs to the category of library approximation, is obtained by partitioning the parameter domain and by using a different truncated Taylor series on each subdomain. In Theorem~\ref{tcount_general_PDE}, we give an upper bound on the number of subdomains (and thus on the size of the library) needed to achieve a prescribed accuracy while using a prescribed number of terms for each truncated Taylor series. The key ingredient needed to obtain such partition is a local error estimate quantifying, for any subdomain, the worst-case error between the exact solution and a local Taylor polynomial. Indeed, once a local error estimate is available, then the partition strategy proposed in \cite{bonito2020nonlinear} can be straightforwardly applied.

In the elliptic PDE setting of~\cite{bonito2020nonlinear}, a local error estimate can be obtained by a simple scaling and shifting argument. This is not possible in the general case considered here, where the only information we have on the function we want to approximate is that it belongs to the class $\cB_{\rho,p}$, and a more refined argument is needed. The main results of this work are thus Theorems \ref{thm:local_v1} and \ref{thm:local_v2} where we derive local error estimates. The upper bound obtained in Theorem~\ref{thm:local_v1} holds under the (strong) assumption that the sequence $\rho$ satisfy $(\rho_j^{-1})_{j\ge 1}\in\ell_1(\mathbb{N})$. This assumption is not needed for the bound of Theorem~\ref{thm:local_v2}, which involves a slightly different sequence than the original sequence $\rho$.

We conclude by mentioning that given a particular function $u\in\cB_{\rho,p}$, a number of terms $m+1$ and an error tolerance $\varepsilon$, the partition considered in the proof of Theorem \ref{tcount_general_PDE} is not optimal in the sense that it does not provide a library $\cL$ of minimal cardinality (with affine spaces of dimension $m$) such that $E_{\cL}(\cM)\le \varepsilon$ for $\cM=\{u(y): \, y\in Y\}$. In particular, we restrict the approximations to (i) piecewise Taylor polynomials and (ii) tensor-product based partitions of $Y$. There are thus opportunities for improvement, for instance by considering other strategies for partitioning the parameter domain as well as other types of spaces (such as piecewise Legendre polynomials or local reduced basis). Nonetheless, the proposed procedure for building a library given $u\in \cB_{\rho,p}$, $m$ and $\varepsilon$ is easy to implement and the library is computationally cheap to construct.

\section*{Acknowledgements}
The authors would like to thank Albert Cohen, Ron DeVore, Guergana Petrova, and Andrea Bonito for helpful discussions.

\bibliographystyle{siam}
\bibliography{bibliography}
 
\end{document}